\documentclass[12pt]{article}
\newcommand{\authorfootnotes}{\renewcommand\thefootnote{\@fnsymbol\c@footnote}}%
\usepackage{amssymb,amsmath, amsfonts, amsthm}

\usepackage{graphicx}
\usepackage{tikz}
\usepackage{enumerate}

\newcommand{\cp}{\,\square\,}

\newcommand{\rotwo}{\rho_2}
\newcommand{\roo}{\rho^{\rm o}}

\tikzstyle{vertex}=[circle, draw, inner sep=0pt, minimum size=6pt]

\usetikzlibrary{arrows}

\newtheorem{thm}{Theorem}[section]
\newtheorem{lem}[thm]{Lemma}

\newtheorem{prop}[thm]{Proposition}

\begin{document}

\title{Packings in bipartite prisms and hypercubes}

\author{
Bo\v{s}tjan Bre\v{s}ar$^{a,b}$
\and
Sandi Klav\v{z}ar$^{a,b,c}$
\and
Douglas F. Rall$^{d}$\\
}

%\lhead{}
%\rhead{}
\date{\today}

\maketitle

\begin{center}
$^a$ Faculty of Natural Sciences and Mathematics, University of Maribor, Slovenia\\
$^b$ Institute of Mathematics, Physics and Mechanics, Ljubljana, Slovenia\\
$^c$ Faculty of Mathematics and Physics, University of Ljubljana, Slovenia\\
$^d$ Department of Mathematics, Furman University, Greenville, SC, USA\\
\end{center}

\begin{abstract}
The $2$-packing number $\rho_2(G)$ of a graph $G$ is the cardinality of a largest $2$-packing of $G$ and the open packing number $\rho^{\rm o}(G)$ is the cardinality of a largest open packing of $G$, where an open packing (resp.~$2$-packing) is a set of vertices in $G$ no two (closed) neighborhoods of which intersect. It is proved that if $G$ is bipartite, then $\rho^{\rm o}(G\cp K_2) = 2\rho_2(G)$. For hypercubes, the lower bounds $\rho_2(Q_n) \ge  2^{n - \lfloor \log n\rfloor -1}$ and $\rho^{\rm o}(Q_n) \ge 2^{n - \lfloor \log (n-1)\rfloor -1}$ are established. These findings are applied to injective colorings of hypercubes. In particular, it is demonstrated that $Q_9$ is the smallest hypercube which is not perfect injectively colorable. It is also proved that $\gamma_t(Q_{2^k}\times H) = 2^{2^k-k}\gamma_t(H)$, where $H$ is an arbitrary graph with no isolated vertices.
\end{abstract}


\noindent
{\bf AMS Subj.\ Class.\ (2020)}: 05C69, 05C76

\maketitle

%%%%%%%%%%%%%%%%%%%%%%%%%%%%%
\section{Introduction}
\label{sec:intro}
%%%%%%%%%%%%%%%%%%%%%%%%%%%%%

For many reasons, hypercubes are ubiquitous in theoretical computer science and in combinatorics. Understanding their structure is therefore a fundamental problem. Although hypercubes have a seemingly simple structure, we quickly encounter very complex problems. For instance, one of them was the middle levels problem, which was successfully dismissed~\cite{mutze-2016}. On the other hand, the problem of determining the domination number of hypercubes is beyond the reach of existing methods.
To date, exact values of $\gamma(Q_n)$ are only known for $n \le 9$, where the value $\gamma(Q_9) = 62$ was obtained in~\cite{ostergard-2001}, and for the following two infinite families.

\begin{thm}{\rm(\cite{harary-1993,wee-1988})}
\label{thm:infinite-families}
If $k\ge 1$, then $\gamma(Q_{2^k-1}) = 2^{2^k-k-1}$ and $\gamma(Q_{2^k}) = 2^{2^k-k}$.
\end{thm}

The values $\gamma(Q_{2^k-1}) = 2^{2^k-k-1}$ can be obtained from the fact that hypercubes $Q_{2^k-1}$ admit 1-perfect codes, in which case the domination number coincides with the cardinality of a 1-perfect code.

The most important variation of the domination number is the total domination number; see a recent monograph~\cite{bookHHH} surveying domination theory with the two invariants in the central role. Roughly speaking, domination operates with closed neighborhoods while total domination with open neighborhoods, which often causes a different behavior of the invariants. However, as proved in~\cite{azarija-2017} by using hypergraph transversals, $\gamma_t(Q_{n+1})= 2\gamma(Q_n)$ for all $n$, which makes the domination number and the total domination number in hypercubes tightly connected. More generally, the authors of~\cite{azarija-2017} proved that $\gamma_t(G\cp K_2)=2\gamma(G)$ as soon as $G$ is a bipartite graph.

The concepts of packing number and open packing number of a graph are often used in domination theory, since they present natural lower bounds on the domination number and the total domination number, respectively, of the graph. The concept of packing was used back in 1975 by Meir and Moon in their classical theorem stating that in a tree the domination number equals the packing number~\cite{mm-1975}. On the other hand, open packing was introduced by Henning and Slater~\cite{hs}, and was later used in~\cite{Rall} to prove a canonical formula for the total domination number of the direct product of two graphs, which holds if one of the factors has the total domination number equal to its open packing number. Similarly as total domination is related to domination, open packing can be regarded as a version of packing in which closed neighborhoods are replaced with open neighborhoods. See~\cite{mohammadi-2019, mojdeh-2020, mojdeh-2022} for some recent studies of (open) packings as well as~\cite{gao-2017} for their application.

Open packings are also related to the so-called injective colorings of graphs, cf.~\cite{pp}. More precisely, an injective coloring of a graph is exactly a partition of its vertex set into open packings. In a recent paper~\cite{bsy}, graphs that admit injective colorings such that each of the color classes is a maximum open packing were considered. While proving this property for hypercubes of some small dimensions, it was also proved for those whose dimension is a power of $2$. Yet, nothing else was known, including whether there exists a hypercube that does not satisfy this property. One of the reasons for the difficulty of this question is that the open packing number (i.e., the cardinality of a maximum open packing) has not been known.

We proceed as follows. In the remainder of this introduction, we provide the definitions and concepts we need for the following. In Section~\ref{sec:bip} we prove that the open packing number of a prism over a bipartite graph $G$ is twice the $2$-packing number of $G$. This result nicely complements~\cite[Theorem 1]{azarija-2017} which states that the total domination number of a prism over a bipartite graph $G$ is twice the domination number of $G$. We also demonstrate that in general, the open packing number of a prism over a graph $G$ can be arbitrary larger that the $2$-packing number of $G$. In Section~\ref{sec:hyp} we prove lower bounds on the $2$-packing number and the open packing number of hypercubes. The bounds are sharp for small dimensions and for two infinite families, but are not sharp in general. In the subsequent section we apply these findings to injective colorings of hypercubes. In particular we demonstrate that $Q_9$ is the smallest hypercube which is not perfect injectively colorable. In the concluding remarks, we give an overview of the known values for the hypercube invariants considered here and also derive the total domination number of the direct product of $Q_{2^k}$ and an arbitrary graph.

%%%%%%%%%%%%%%%%%%%%%%%%%%%%%
\subsection{Preliminaries}
\label{sec:prelim}
%%%%%%%%%%%%%%%%%%%%%%%%%%%%%

Let $G = (V(G), E(G))$ be a graph and $x\in V(G)$. The {\em open neighborhood} $N(x)$ is the set of vertices adjacent to $x$ and the {\em closed neighborhood} is $N[x] = N(x)\cup \{x\}$. A set $D\subseteq V(G)$ is a {\em dominating set} of $G$ if each vertex of $V(G)\setminus D$ has a neighbor in $D$. The cardinality of a smallest dominating set of $G$ is the {\em domination number} $\gamma(G)$ of $G$. Similarly, $D\subseteq V(G)$ is a {\em total dominating set} of $G$ if each vertex of $V(G)$ has a neighbor in $D$. The cardinality of a smallest dominating set of $G$ is the {\em total domination number} $\gamma_t(G)$ of $G$.

Let $X\subseteq V(G)$. Then $X$ is a {\em $2$-packing} of $G$ if $N[x] \cap N[y] = \emptyset$ for every pair of distinct vertices $x,y\in X$. Similarly, if $N(x) \cap N(y) = \emptyset$ for every pair of distinct vertices $x,y\in X$, then $X$ is an  {\em open packing} of $G$. The cardinality of a largest $2$-packing of $G$ is the {\em $2$-packing number} $\rotwo(G)$ of $G$ and the cardinality of a largest open packing of $G$ is the {\em open packing number} $\roo(G)$ of $G$. By a {\em $\rotwo$-set} of $G$ we mean a $2$-packing of $G$ of cardinality $\rotwo(G)$. A {\em $\roo$-set} of $G$ is defined analogously.

If $X$ is a $2$-packing such that  $V(G) = \cup_{x\in X} N[x]$ then we say that $X$ is a {\em $1$-perfect code} of $G$. In domination theory, 1-perfect codes are known as {\em efficient dominating sets}, see~\cite[Chapter 9]{bookHHH} and~\cite{KP}. Since $\gamma(G) \ge \rotwo(G)$ for every graph $G$, if $X$ is a $1$-perfect code of $G$, then $X$ is also a dominating set of $G$. This observation leads to the following well known fact.

\begin{prop}
\label{prop:gamma=rotwo}
If $G$ admits a $1$-perfect code, then $\gamma(G) = \rotwo(G)$. If in addition $G$ is $r$-regular, then $\gamma(G) = \rotwo(G)=\frac{n(G)}{r+1}$.
\end{prop}

The \emph{Cartesian product} $G\cp H$ of graphs $G$ and $H$ is the graph whose vertex set is $V(G) \times V(H)$, and two vertices $(g_1,h_1)$ and $(g_2,h_2)$ are adjacent in $G \cp H$ if either $g_1=g_2$ and $h_1h_2$ is an edge in $H$ or $h_1=h_2$ and $g_1g_2$ is an edge in $G$.  For a vertex $g$ of $G$, the subgraph of $G\cp H$ induced by the set $\{(g,h):\, h\in V(H)\}$ is an \emph{$H$-fiber} and is denoted by $^g\!H$.  Similarly, for $h \in H$, the \emph{$G$-fiber}, $G^h$, is the subgraph induced by $\{(g,h):\, g\in V(G)\}$. Cartesian product is commutative and associative. The {\em hypercube} of dimension $n$, or the {\em $n$-cube}, is isomorphic to $K_2\cp\cdots\cp K_2$, where there are $n$ factors $K_2$, and is denoted by $Q_n$. The equality $Q_n=Q_{n-1}\cp K_2$ will be used (at least implicitly) several times in the paper. Finally, the \emph{direct product} $G\times H$ of graphs $G$ and $H$ has the vertex set $V(G) \times V(H)$, and two vertices $(g_1,h_1)$ and $(g_2,h_2)$ are adjacent in $G \times H$ if $g_1g_2$ is an edge in $G$ and $h_1h_2$ is an edge in $H$.

%%%%%%%%%%%%%%%%%%%%%%%%%%%%%%%%%%%%%%%%
\section{Packing vs.\ open packing in bipartite prisms}
\label{sec:bip}
%%%%%%%%%%%%%%%%%%%%%%%%%%%%%%%%%%%%%%%%

In~\cite{azarija-2017} it was proved that if $G$ is a bipartite graph, then $\gamma_t(G\cp K_2)=2\gamma(G)$. In this section we prove an analogous result that connects the open packing number and the packing number.

We begin with the following simple lemma, which holds in all graphs.

\begin{lem}
\label{lem:lower-bound}
If $G$ is a graph, then $\roo(G\cp K_2)\ge 2\rotwo(G)$.
\end{lem}
\begin{proof}
Let $G$ be a graph, and let $P$ be a $\rotwo$-set of $G$. Then $P\times V(K_2)$ is an open packing of $G\cp K_2$, hence the result.
\end{proof}

In general, $\roo(G\cp K_2)$ can be arbitrary larger than $2 \rotwo(G)$. For an example consider the family of graphs $G_k$, $k\ge 1$, defined as follows. $G_k$ contains $2k$ disjoint cycles $C_5$ connected in a row by an edge between two consecutive $5$-cycles. This informal definition of $G_k$ should be clear from Fig.~\ref{fig:G2} where $G_2\cp K_2$ is drawn. As an arbitrary packing of $G_k$ contains at most one vertex of each $C_5$ we infer that $\rotwo(G_k) = 2k$. On the other hand, repeating the pattern as shown in Fig.~\ref{fig:G2} for $k=2$, we get $\roo(G_k\cp K_2) \ge 5k$.

\begin{figure}[ht!]
\begin{center}
\begin{tikzpicture}[scale=0.78,style=thick,x=1cm,y=1cm]
\def\vr{3pt}

\begin{scope} % 1
\coordinate(1) at (0,0);
\coordinate(2) at (1,0);
\coordinate(3) at (2,0);
\coordinate(4) at (3,0);
\coordinate(5) at (4,0);
\coordinate(6) at (0,1);
\coordinate(7) at (1,1);
\coordinate(8) at (2,1);
\coordinate(9) at (3,1);
\coordinate(10) at (4,1);
% \edges		
\draw (1) -- (5) -- (10) -- (6) -- (1);
\draw (1) -- (6);
\draw (2) -- (7);
\draw (3) -- (8);
\draw (4) -- (9);
\draw (5) -- (10);
\draw (4,0) -- (5,0);
\draw (4,1) -- (5,1);
\draw (0,1) .. controls (1,2) and (3,2) .. (4,1);
\draw (0,0) .. controls (1,-1) and (3,-1) .. (4,0);
%  vertices
\foreach \i in {1,2,...,10}
{
\draw(\i)[fill=white] circle(\vr);
}
\draw(4)[fill=black] circle(\vr);
\draw(6)[fill=black] circle(\vr);
\draw(7)[fill=black] circle(\vr);
% text
\end{scope}
		
\begin{scope}[xshift=5cm, yshift=0cm] % 2
\coordinate(1) at (0,0);
\coordinate(2) at (1,0);
\coordinate(3) at (2,0);
\coordinate(4) at (3,0);
\coordinate(5) at (4,0);
\coordinate(6) at (0,1);
\coordinate(7) at (1,1);
\coordinate(8) at (2,1);
\coordinate(9) at (3,1);
\coordinate(10) at (4,1);
% \edges		
\draw (1) -- (5) -- (10) -- (6) -- (1);
\draw (1) -- (6);
\draw (2) -- (7);
\draw (3) -- (8);
\draw (4) -- (9);
\draw (5) -- (10);
\draw (4,0) -- (5,0);
\draw (4,1) -- (5,1);
\draw (0,1) .. controls (1,2) and (3,2) .. (4,1);
\draw (0,0) .. controls (1,-1) and (3,-1) .. (4,0);
%  vertices
\foreach \i in {1,2,...,10}
{
\draw(\i)[fill=white] circle(\vr);
}
\draw(3)[fill=black] circle(\vr);
\draw(8)[fill=black] circle(\vr);
% text
\end{scope}

\begin{scope}[xshift=10cm, yshift=0cm] % 3
\coordinate(1) at (0,0);
\coordinate(2) at (1,0);
\coordinate(3) at (2,0);
\coordinate(4) at (3,0);
\coordinate(5) at (4,0);
\coordinate(6) at (0,1);
\coordinate(7) at (1,1);
\coordinate(8) at (2,1);
\coordinate(9) at (3,1);
\coordinate(10) at (4,1);
% \edges		
\draw (1) -- (5) -- (10) -- (6) -- (1);
\draw (1) -- (6);
\draw (2) -- (7);
\draw (3) -- (8);
\draw (4) -- (9);
\draw (5) -- (10);
\draw (4,0) -- (5,0);
\draw (4,1) -- (5,1);
\draw (0,1) .. controls (1,2) and (3,2) .. (4,1);
\draw (0,0) .. controls (1,-1) and (3,-1) .. (4,0);
%  vertices
\foreach \i in {1,2,...,10}
{
\draw(\i)[fill=white] circle(\vr);
}
\draw(4)[fill=black] circle(\vr);
\draw(6)[fill=black] circle(\vr);
\draw(7)[fill=black] circle(\vr);
% text
\end{scope}

\begin{scope}[xshift=15cm, yshift=0cm] % 4
\coordinate(1) at (0,0);
\coordinate(2) at (1,0);
\coordinate(3) at (2,0);
\coordinate(4) at (3,0);
\coordinate(5) at (4,0);
\coordinate(6) at (0,1);
\coordinate(7) at (1,1);
\coordinate(8) at (2,1);
\coordinate(9) at (3,1);
\coordinate(10) at (4,1);
% \edges		
\draw (1) -- (5) -- (10) -- (6) -- (1);
\draw (1) -- (6);
\draw (2) -- (7);
\draw (3) -- (8);
\draw (4) -- (9);
\draw (5) -- (10);
\draw (0,1) .. controls (1,2) and (3,2) .. (4,1);
\draw (0,0) .. controls (1,-1) and (3,-1) .. (4,0);
%  vertices
\foreach \i in {1,2,...,10}
{
\draw(\i)[fill=white] circle(\vr);
}
\draw(3)[fill=black] circle(\vr);
\draw(8)[fill=black] circle(\vr);
% text
\end{scope}
\end{tikzpicture}
\caption{An open packing in $G_2\cp K_2$}
	\label{fig:G2}
\end{center}
\end{figure}
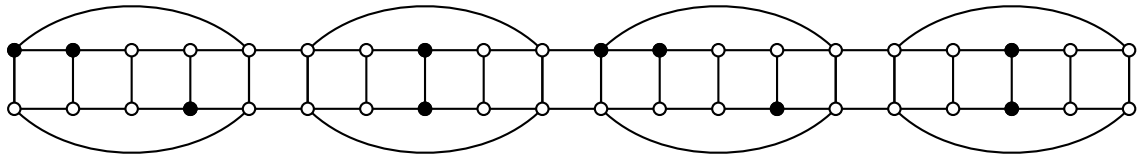

For bipartite graphs, however, the above phenomena cannot occur as the main result of this section asserts.

\begin{thm}
\label{thm:bip-prisms}
If $G$ is a bipartite graph, then $\roo(G\cp K_2)=2\rotwo(G)$.
\end{thm}
\begin{proof}
Let $G$ be a bipartite graph with parts $A$ and $B$ forming the natural partition of $V(G)$.
By Lemma~\ref{lem:lower-bound}, we have $\roo(G\cp K_2)\ge 2\rotwo(G)$.
To prove the reversed inequality, consider an open packing $O$ in $G\cp K_2$ such that $|O|=\roo(G\cp K_2)$. We will show that $O$ can be transformed into an open packing $O'$ of the form $P'\times V(K_2)$, where $P'$ is a subset of $V(G)$. (Clearly, the latter also implies that $P'$ is a $2$-packing.)
Note that $O$ can be presented as the disjoint union $I\cup R$, where $I$ is the set of vertices that are isolated in the subgraph of $G\cp K_2$ induced by $O$, while $R$ is the set of vertices that have exactly one neighbor in $O$. Clearly, at least one of the sets $I$ or $R$ is non-empty. Set $V(K_2)=\{1,2\}$, and let $I_i=I\cap V(G^i)$ and $R_i=R\cap V(G^i)$ for all $i\in [2]$. In addition, let $I_i^A=\{(u,i)\in I_i:\, u\in A\}$, $I_i^B=\{(u,i)\in I_i:\, u\in B\}$ for $i\in [2]$, and similarly let $R_i^A=\{(u,i)\in R_i:\, u\in A\}$, $R_i^B=\{(u,i)\in R_i:\, u\in B\}$ for $i\in [2]$.   Next, we compare the two quantities $|I_1^A|+|I_2^B|$ and $|I_2^A|+|I_1^B|$. We may assume with no loss of generality that $|I_1^A|+|I_2^B|\ge|I_2^A|+|I_1^B|$. Now, the announced transformation of $O$ to $O'$ is defined as follows:
\begin{itemize}
\item if $(u,t)\in I_1^A\cup I_2^B$, then let $\{u\}\times V(K_2)\subseteq O'$;
\item if $(u,t)\in I_2^A\cup I_1^B$, then let $(\{u\}\times V(K_2))\cap O'=\emptyset$;
\item if $(u,1)\in R_1$ and  $(u,2)\in R_2$, then let $\{u\}\times V(K_2)\subseteq O'$;
\item if $(u,1)\in R_1^A$ and  $(v,1)\in R_1^B$, where $uv\in E(G)$, then let  $\{u\}\times V(K_2)\subseteq O'$ and  $(\{v\}\times V(K_2))\cap O'=\emptyset$;
\item if $(u,2)\in R_2^A$ and  $(v,2)\in R_2^B$, where $uv\in E(G)$, then let  $\{v\}\times V(K_2)\subseteq O'$ and  $(\{u\}\times V(K_2))\cap O'=\emptyset$.
\end{itemize}
We claim that $|O'|\ge |O|$. Indeed, the first two rows in the above transformation  show that for every vertex $(u,t)\in I_1^A\cup I_2^B$ we get two vertices in $O'$, while for every vertex $(u,t)\in I_2^A\cup I_1^B$ we get no vertices in $O'$, yet $|I_1^A\cup I_2^B|>|I_2^A\cup I_1^B|$ by the earlier assumption. By the last three rows of the above transformation, every pair of vertices in $R$ is replaced by two vertices in $O'$. This altogether implies that $|O'|\ge |O|$, so it remains to prove that $O'$ is an open packing in $G\cp K_2$.

If $(u,1)\in I_1^A$ and $(v,1)\in I_1^A$, then $d_G(u,v)\ge 4$, because the vertices belong to $O$, which is an open packing, and $u$ and $v$ are both in $A$. Thus vertices in $\{u\}\times V(K_2)$ will be at distance at least $4$ from the vertices in $\{v\}\times V(K_2)$. By symmetry, we get the same conclusion for vertices $(u,2)\in I_2^B$ and $(v,2)\in I_2^B$. If $(u,1)\in I_1^A$ and $(v,2)\in I_2^B$, then $d_G(u,v)\ge 3$, because $u$ and $v$ belong to different parts, $A$ and $B$ respectively, of the bipartition of $V(G)$ and they belong to $O$, which is an open packing. Thus, vertices in $\{u\}\times V(K_2)$ will be at distance at least $3$ from the vertices in $\{v\}\times V(K_2)$, as desired. Clearly, if $(u,t)\in I_1^A\cup I_2^B$, then $d_G(u,v)\ge 3$ for any $v\in V(G)$ such that $\{(v,1),(v,2)\}\subset R$. This yields that vertices in $\{u\}\times V(K_2)$ will be at distance at least $3$ from the vertices in $\{v\}\times V(K_2)$. If $(u,1)\in I_1^A$ and $(v,1)\in R_1^A$, we have $d_G(u,v)\ge 4$. On the other hand, if $(u,1)\in I_1^A$ and $(v,2)\in R_2^B$ we have $d_G(u,v)\ge 3$. In either case, the corresponding vertices in $O'$ are at least three apart. By symmetry, we can find that for vertices in $I_2^B$ and vertices in $R_1^A\cup R_2^B$ their distances are sufficiently large so that the corresponding $K_2$-fibers that are in $O'$ will be at distance at least $3$. This completes the proof that the distance between the vertices in $O'$ that appear in the first row of the above transformation to all other vertices in $O'$ will be at least $3$, except of course for two vertices in $O'$ that belong to the same $K_2$-fiber and are adjacent.

Vertices of $O'$ that appear in the third row of the transformation %(which actually does not change from $O$)
remain at distance at least $3$ from all other vertices in $O'$ (with the clear exception of two adjacent such vertices). Therefore, it remains to consider the vertices in $O'$ that appear in the last two rows of the above transformation. Suppose there are two vertices in $R_1^A$ (and a similar argument can be applied if they are in $R_2^B$), say, $(u,1)$ and $(v,1)$, which are not adjacent. Then $d_G(u,v)\ge 4$, and so $\{u\}\times V(K_2)$ will be at distance at least $4$ from the vertices in $\{v\}\times V(K_2)$ (by symmetry, the same conclusion applies if $(u,2)$ and $(v,2)$ are in $R_2^B$). Finally, let $(u,1)\in R_1^A$ and $(v,2)\in R_2^B$. Since $O$ is an open packing, we have $d_G(u,v)>1$, and since they are in different parts of the bipartition, we get $d_G(u,v)\ge 3$. We derive that $\{u\}\times V(K_2)$ will be at distance at least $3$ from the vertices in $\{v\}\times V(K_2)$, which concludes the proof that $O'$ is an open packing. Since $|O|=\roo(G\cp K_2)$ and $|O'|\ge |O|$, we derive $|O'|=|O|=\roo(G\cp K_2)$. In addition, there exists a set $P'\subset V(G)$ such that $O'=P'\times [2]$, where $P'$ is a $2$-packing of $G$. Hence, $|P'|\le \rotwo(G)$, and so $|O'|=2|P'|\le 2\rotwo(G)$, implying $\roo(G\cp K_2)\le 2\rotwo(G)$.
\end{proof}

%%%%%%%%%%%%%%%%%%%%%%%%%%%%%
\section{(Open) packings in hypercubes}
%%%%%%%%%%%%%%%%%%%%%%%%%%%%%
\label{sec:hyp}

The following lemma follows by observing that the restriction of a 2-packing in $G\cp K_2$ to a $G$-layer is a 2-packing of that layer.

\begin{lem}
\label{lem:prism-rotwo}
If $G$ is a graph, then $\rotwo(G\cp K_2) \le 2 \rotwo(G)$.
\end{lem}

We can now bound $\rotwo$ and $\roo$ of hypercubes as follows.

\begin{thm}
\label{thm:lower-bound}
If $n\ge 2$, then
\begin{enumerate}
\item[(i)] $\rotwo(Q_n) \ge  2^{n - \lfloor \log n\rfloor -1}$\quad and
\item[(ii)] $\roo(Q_n) \ge 2^{n - \lfloor \log (n-1)\rfloor -1}$.
\end{enumerate}
\end{thm}
\begin{proof}
(i) Suppose first that $n = 2^k - 1$, where $k\ge 2$. As already mentioned, in these cases $Q_n$ admits a $1$-perfect code, say $S$. Then $|S| = 2^{2^k - 1}/2^k = 2^{2^k - k - 1}$ and consequently
$$\rotwo(Q_n) = |S| = 2^{2^k - k - 1} = 2^{2^k -1 - (k - 1) - 1} = 2^{n - \lfloor \log n\rfloor -1}\,.$$
Consider now the hypercubes $Q_n$, where $k\ge 3$ and $2^{k-1} - 1 < n < 2^k - 1$. In particular, if $n = 2^k - 2$, then since $Q_{2^k - 1} = Q_{2^k - 2} \cp K_2$, Lemma~\ref{lem:prism-rotwo} implies that
$$\rotwo(Q_{n}) = \rotwo(Q_{2^k - 2})
\ge \frac{1}{2} \rotwo(Q_{2^k - 1})
= 2^{2^k-k-2} = 2^{2^k-2 - (k-1)-1}
= 2^{n - \lfloor \log n\rfloor -1}
\,.$$
Inductively applying the lemma, the result holds for all $n$ such that $2^{k-1} - 1 < n < 2^k - 1$. Therefore, (i) holds for all $n\ge 2$.

(ii) Applying Theorem~\ref{thm:bip-prisms} and (i), we have
$$\roo(Q_n) = 2\rotwo(Q_{n-1}) \ge
2\cdot 2^{(n-1) - \lfloor \log (n-1)\rfloor -1} = 2^{n - \lfloor \log (n-1)\rfloor -1}$$
for all $n\ge 2$ and we are done.
\end{proof}

If $n\le 7$, then equality holds in Theorem~\ref{thm:lower-bound}(i). The cases when $n\in \{2,3,4\}$ can be easily argued by case analysis. The equality in  cases when $n\in \{5,6\}$ then follow by combining Lemma~\ref{lem:prism-rotwo} and Theorem~\ref{thm:lower-bound}(i). For $n = 7$, the equality holds because $Q_7$ has a $1$-perfect code. One is thus tempted to conjecture that the lower bound in Theorem~\ref{thm:lower-bound}(i) holds for all $n$. However, with the help of a computer, we found the set
\begin{align*}
T =\ & \{00000000, 00001110, 00110010, 00111100, 01010110, 01011000, \\ & \ 01100100, 01101001, 01111111, 10010100, 10100101, 10101011, \\
& \ 11000111, 11001100, 11011011, 11100010, 11110001\}
\end{align*}
which is a 2-packing in $Q_8$ with $|T| = 17$, hence $\rotwo(Q_8) \ge 17$. By Theorem~\ref{thm:bip-prisms}, this in turn implies that
$\roo(Q_9) \ge 34$. Hence also the lower bound in Theorem~\ref{thm:lower-bound}(ii) is not sharp in general. It is sharp however for all $n\le 8$ because the lower bound in Theorem~\ref{thm:lower-bound}(i) is sharp for $n\le 7$ and because of Theorem~\ref{thm:bip-prisms}.  Furthermore, by using Theorem~\ref{thm:bip-prisms} and the fact
that the lower bound in Theorem~\ref{thm:lower-bound}(i) is sharp when $n=2^k - 1$, it follows that the lower bound in Theorem~\ref{thm:lower-bound}(ii) is sharp for
each value of $n$ that is a power of $2$.

%%%%%%%%%%%%%%%%%%%%%%%%%%
\section{Application to injective colorings}
\label{sec:injective}
%%%%%%%%%%%%%%%%%%%%%%%%%%%

An {\em injective coloring} of a graph $G$ is a partition of the vertex set of $G$ into open packings. The {\em injective chromatic number}, $\chi_i(G)$, of $G$ is the minimum cardinality of an injective coloring in $G$. The concept was introduced by Hahn, Kratochv\'{i}l, \v{S}ir\'{a}\v{n} and Sotteau~\cite{hkss} back in 2002, and has been considered by a number of authors, cf.~\cite{brause-2022, chen-2012}. In the recent paper~\cite{bsy}, graphs that admit special types of injective colorings were considered: a graph $G$ is a \emph{perfect injectively colorable graph} if it has an injective coloring in which every color class forms a $\roo$-set of $G$.
The authors of~\cite{bsy} considered hypercubes that are perfect injectively colorable. They noticed that such are the hypercubes $Q_n$, where $n\in [5]$, and proved that for all $k\in \mathbb{N}$, the hypercube $Q_{2^k}$ is a perfect injectively colorable graph.
Apart from the mentioned cases, it was asked in~\cite[Problem 1]{bsy} in which other dimensions the hypercube is perfect injectively colorable.
Since an answer to the question is closely related to computing the value of the open packing number of hypercubes, it was also asked in~\cite[Problem 2]{bsy} what is the value of $\roo(Q_n)$ for $n\ge 6$.

\begin{figure}[ht!]
\begin{center}
\begin{tikzpicture}[scale=1.2,style=thick,x=1cm,y=1cm]
\def\vr{2pt}

\begin{scope} % 1
\coordinate(1) at (0,0);
\coordinate(2) at (2,0);
\coordinate(3) at (2,2);
\coordinate(4) at (0,2);
\coordinate(5) at (0.6,0.6);
\coordinate(6) at (1.4,0.6);
\coordinate(7) at (1.4,1.4);
\coordinate(8) at (0.6,1.4);
% \edges		
\draw (1) -- (2) -- (3) -- (4) -- (1);
\draw (5) -- (6) -- (7) -- (8) -- (5);
\draw (1) -- (5);
\draw (2) -- (6);
\draw (3) -- (7);
\draw (4) -- (8);
%  vertices
\foreach \i in {1,2,...,8}
{
\draw(\i)[fill=white] circle(\vr);
}
% text
\node at (-0.2,0) {$3$};
\node at (2.2,0) {$4$};
\node at (-0.2,2) {$1$};
\node at (2.2,2) {$2$};
\node at (0.4,0.65) {$2$};
\node at (0.4,1.35) {$4$};
\node at (1.6,0.65) {$1$};
\node at (1.6,1.35) {$3$};
\draw [line width=0.5mm, dashed] (2.2,1) -- (2.8,1);
\draw [line width=0.5mm, dashed] (4,-0.2) -- (4,-0.8);
\draw [line width=0.5mm, dashed] (2.2,-2) -- (2.8,-2);
\draw [line width=0.5mm, dashed] (1,-0.2) -- (1,-0.8);
\draw [line width=0.5mm, dashed] (1,2.2) .. controls (2,3.3) and (6,3.3) .. (7,2.2);
\end{scope}
		
\begin{scope}[xshift=3cm, yshift=0cm] % 2
\coordinate(1) at (0,0);
\coordinate(2) at (2,0);
\coordinate(3) at (2,2);
\coordinate(4) at (0,2);
\coordinate(5) at (0.6,0.6);
\coordinate(6) at (1.4,0.6);
\coordinate(7) at (1.4,1.4);
\coordinate(8) at (0.6,1.4);
% \edges		
\draw (1) -- (2) -- (3) -- (4) -- (1);
\draw (5) -- (6) -- (7) -- (8) -- (5);
\draw (1) -- (5);
\draw (2) -- (6);
\draw (3) -- (7);
\draw (4) -- (8);
%  vertices
\foreach \i in {1,2,...,8}
{
\draw(\i)[fill=white] circle(\vr);
}
\draw [line width=0.5mm, dashed] (1,2.2) .. controls (2,3.3) and (6,3.3) .. (7,2.2);

\node at (-0.2,0) {$7$};
\node at (2.2,0) {$8$};
\node at (-0.2,2) {$5$};
\node at (2.2,2) {$6$};
\node at (0.4,0.65) {$6$};
\node at (0.4,1.35) {$8$};
\node at (1.6,0.65) {$5$};
\node at (1.6,1.35) {$7$};
\end{scope}

\begin{scope}[xshift=0cm, yshift=-3cm] % 3
\coordinate(1) at (0,0);
\coordinate(2) at (2,0);
\coordinate(3) at (2,2);
\coordinate(4) at (0,2);
\coordinate(5) at (0.6,0.6);
\coordinate(6) at (1.4,0.6);
\coordinate(7) at (1.4,1.4);
\coordinate(8) at (0.6,1.4);
% \edges		
\draw (1) -- (2) -- (3) -- (4) -- (1);
\draw (5) -- (6) -- (7) -- (8) -- (5);
\draw (1) -- (5);
\draw (2) -- (6);
\draw (3) -- (7);
\draw (4) -- (8);
%  vertices
\foreach \i in {1,2,...,8}
{
\draw(\i)[fill=white] circle(\vr);
}
\draw [line width=0.5mm, dashed] (1,-0.2) .. controls (2,-1.3) and (6,-1.3) .. (7,-0.2);
\node at (-0.2,0) {$8$};
\node at (2.2,0) {$6$};
\node at (-0.2,2) {$7$};
\node at (2.2,2) {$5$};
\node at (0.4,0.65) {$5$};
\node at (0.4,1.35) {$6$};
\node at (1.6,0.65) {$7$};
\node at (1.6,1.35) {$8$};
\end{scope}

\begin{scope}[xshift=3cm, yshift=-3cm] % 4
\coordinate(1) at (0,0);
\coordinate(2) at (2,0);
\coordinate(3) at (2,2);
\coordinate(4) at (0,2);
\coordinate(5) at (0.6,0.6);
\coordinate(6) at (1.4,0.6);
\coordinate(7) at (1.4,1.4);
\coordinate(8) at (0.6,1.4);
% \edges		
\draw (1) -- (2) -- (3) -- (4) -- (1);
\draw (5) -- (6) -- (7) -- (8) -- (5);
\draw (1) -- (5);
\draw (2) -- (6);
\draw (3) -- (7);
\draw (4) -- (8);
%  vertices
\foreach \i in {1,2,...,8}
{
\draw(\i)[fill=white] circle(\vr);
}
% text
\node at (-0.2,0) {$4$};
\node at (2.2,0) {$2$};
\node at (-0.2,2) {$3$};
\node at (2.2,2) {$1$};
\node at (0.4,0.65) {$1$};
\node at (0.4,1.35) {$2$};
\node at (1.6,0.65) {$3$};
\node at (1.6,1.35) {$4$};
\draw [line width=0.5mm, dashed] (1,-0.2) .. controls (2,-1.3) and (6,-1.3) .. (7,-0.2);
\end{scope}

\begin{scope}[xshift=6cm, yshift=0cm] % 5
\coordinate(1) at (0,0);
\coordinate(2) at (2,0);
\coordinate(3) at (2,2);
\coordinate(4) at (0,2);
\coordinate(5) at (0.6,0.6);
\coordinate(6) at (1.4,0.6);
\coordinate(7) at (1.4,1.4);
\coordinate(8) at (0.6,1.4);
% \edges		
\draw (1) -- (2) -- (3) -- (4) -- (1);
\draw (5) -- (6) -- (7) -- (8) -- (5);
\draw (1) -- (5);
\draw (2) -- (6);
\draw (3) -- (7);
\draw (4) -- (8);
%  vertices
\foreach \i in {1,2,...,8}
{
\draw(\i)[fill=white] circle(\vr);
}
\draw [line width=0.5mm, dashed] (2.2,1) -- (2.8,1);
\draw [line width=0.5mm, dashed] (4,-0.2) -- (4,-0.8);
\draw [line width=0.5mm, dashed] (2.2,-2) -- (2.8,-2);
\draw [line width=0.5mm, dashed] (1,-0.2) -- (1,-0.8);
\node at (-0.2,0) {$6$};
\node at (2.2,0) {$5$};
\node at (-0.2,2) {$8$};
\node at (2.2,2) {$7$};
\node at (0.4,0.65) {$7$};
\node at (0.4,1.35) {$5$};
\node at (1.6,0.65) {$8$};
\node at (1.6,1.35) {$6$};
\end{scope}

\begin{scope}[xshift=9cm, yshift=0cm] % 6
\coordinate(1) at (0,0);
\coordinate(2) at (2,0);
\coordinate(3) at (2,2);
\coordinate(4) at (0,2);
\coordinate(5) at (0.6,0.6);
\coordinate(6) at (1.4,0.6);
\coordinate(7) at (1.4,1.4);
\coordinate(8) at (0.6,1.4);
% \edges		
\draw (1) -- (2) -- (3) -- (4) -- (1);
\draw (5) -- (6) -- (7) -- (8) -- (5);
\draw (1) -- (5);
\draw (2) -- (6);
\draw (3) -- (7);
\draw (4) -- (8);
%  vertices
\foreach \i in {1,2,...,8}
{
\draw(\i)[fill=white] circle(\vr);
}
% text
\node at (-0.2,0) {$2$};
\node at (2.2,0) {$1$};
\node at (-0.2,2) {$4$};
\node at (2.2,2) {$3$};
\node at (0.4,0.65) {$3$};
\node at (0.4,1.35) {$1$};
\node at (1.6,0.65) {$4$};
\node at (1.6,1.35) {$2$};
\end{scope}

\begin{scope}[xshift=6cm, yshift=-3cm] % 7
\coordinate(1) at (0,0);
\coordinate(2) at (2,0);
\coordinate(3) at (2,2);
\coordinate(4) at (0,2);
\coordinate(5) at (0.6,0.6);
\coordinate(6) at (1.4,0.6);
\coordinate(7) at (1.4,1.4);
\coordinate(8) at (0.6,1.4);
% \edges		
\draw (1) -- (2) -- (3) -- (4) -- (1);
\draw (5) -- (6) -- (7) -- (8) -- (5);
\draw (1) -- (5);
\draw (2) -- (6);
\draw (3) -- (7);
\draw (4) -- (8);
%  vertices
\foreach \i in {1,2,...,8}
{
\draw(\i)[fill=white] circle(\vr);
}
% text
\node at (-0.2,0) {$1$};
\node at (2.2,0) {$3$};
\node at (-0.2,2) {$2$};
\node at (2.2,2) {$4$};
\node at (0.4,0.65) {$4$};
\node at (0.4,1.35) {$3$};
\node at (1.6,0.65) {$2$};
\node at (1.6,1.35) {$1$};
\end{scope}

\begin{scope}[xshift=9cm, yshift=-3cm] % 8
\coordinate(1) at (0,0);
\coordinate(2) at (2,0);
\coordinate(3) at (2,2);
\coordinate(4) at (0,2);
\coordinate(5) at (0.6,0.6);
\coordinate(6) at (1.4,0.6);
\coordinate(7) at (1.4,1.4);
\coordinate(8) at (0.6,1.4);
% \edges		
\draw (1) -- (2) -- (3) -- (4) -- (1);
\draw (5) -- (6) -- (7) -- (8) -- (5);
\draw (1) -- (5);
\draw (2) -- (6);
\draw (3) -- (7);
\draw (4) -- (8);
%  vertices
\foreach \i in {1,2,...,8}
{
\draw(\i)[fill=white] circle(\vr);
}
\node at (-0.2,0) {$5$};
\node at (2.2,0) {$7$};
\node at (-0.2,2) {$6$};
\node at (2.2,2) {$8$};
\node at (0.4,0.65) {$8$};
\node at (0.4,1.35) {$7$};
\node at (1.6,0.65) {$6$};
\node at (1.6,1.35) {$5$};
\end{scope}		
	
\end{tikzpicture}
\caption{Partition of $V(Q_6)$ into (maximum) $2$-packings of $Q_6$.}
	\label{fig:Q6}
\end{center}
\end{figure}
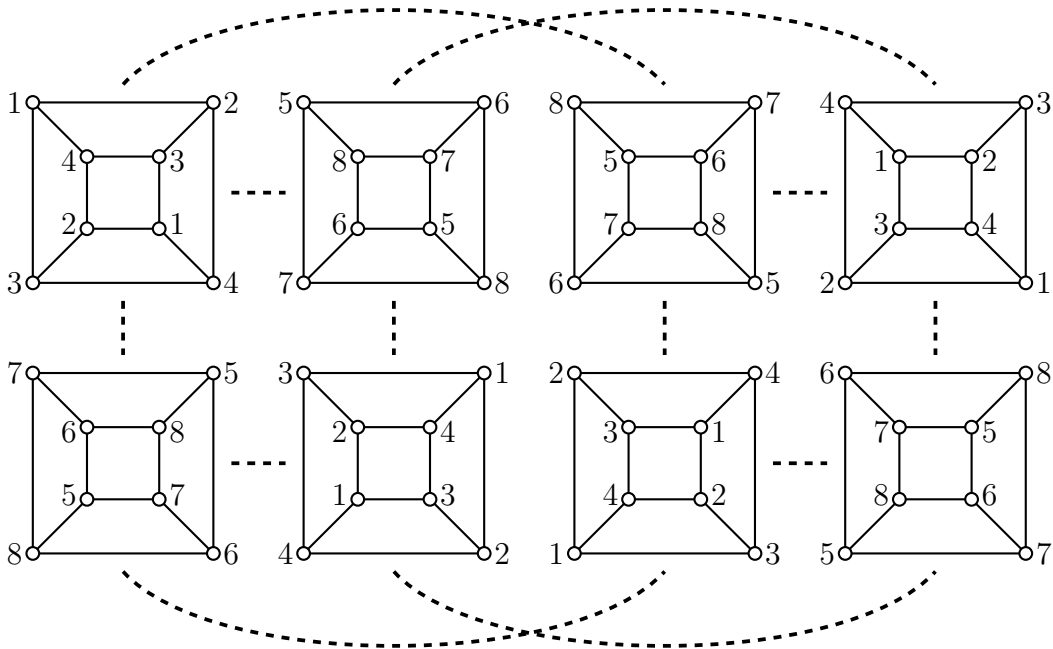

In this note, we give some partial answers to the above two questions. One can easily find that $\rho_2(Q_5)=4$, which by Theorem~\ref{thm:bip-prisms} implies that $\roo(Q_6)=8$. In addition, Fig.~\ref{fig:Q6} shows a maximum $2$-packing of $Q_6$ of cardinality $8$, where vertices of an arbitrary color in $[8]$ form a maximum $2$-packing. This gives, again by Theorem~\ref{thm:bip-prisms}, that $\roo(Q_7)=16$. In addition, recall that $\roo(Q_8)=32$, which follows from the fact that $Q_7$ has a perfect code.
Now, by the observation from Section~\ref{sec:hyp}, we have $\rotwo(Q_8)\ge 17$. On the other hand, we claim that $\rotwo(Q_8)\le 30$. Suppose to the contrary that $\rotwo(Q_8)> 30$, and let $P$ be a $\rotwo$-set of $Q_8$. Then, partitioning $V(Q_8)$ into $Q$ and $Q'$, each of which induces $Q_7$, we infer that either $|Q\cap P|$ or $|Q'\cap P|$ is equal to $16$. We may assume that $|Q\cap P|=16$, and noting that $Q\cap P$ is a $2$-packing of $Q_7$, this implies that $Q\cap P$ corresponds to a perfect code of $Q_7$, thus $Q\cap P$ is a dominating set of $Q$. This in turn implies that every vertex in $Q'$ is at distance at most $2$ from a vertex in $Q\cap P$, which yields that $P=Q\cap P$, and so $|P|=16$, a contradiction proving that $\rotwo(Q_8)\le 30$. Now, using Theorem~\ref{thm:bip-prisms}, we get $34\le \roo(Q_9)\le 60$. In particular, $\roo(Q_9)$ is not a power of $2$, which readily implies that $Q_9$ does not admit a partition into $\roo$-sets, and is consequently not a perfect  injectively colorable graph. On the other hand, refer to Fig.~\ref{fig:Q6} again, which shows a coloring of $Q_6$ in which each color class is a $2$-packing of cardinality $\rotwo(Q_6)$. By applying Theorem~\ref{thm:bip-prisms} and the first part of its proof, one can construct an injective coloring of $Q_7$ in which each color class is a open packing of cardinality $\roo(Q_7)$. Therefore, $Q_7$ is perfect injectively colorable graph.

Summarizing the above, hypercubes $Q_n$, where $n\le 8$, are perfect injectively colorable graphs, and so $Q_9$ is the first instance of a hypercube, which is not in this class of graphs. %\texttt{Any other hypercube for which we know if it is (not) in the class?}

\section{Concluding remarks}

Table~\ref{tab:hypercubes} presents values or bounds on the main domination and packing invariants in hypercubes $Q_n$, for all $n, n\le 9$. The values for $\gamma$ and $\gamma_t$ have been known earlier, while some of the values and bounds for $\rho_2$ and $\roo$ have been obtained in this paper.

\begin{table}[ht!]
\label{tab:hypercubes}
\begin{center}
\begin{tabular}{|l|c|c|c|c|c|c|c|c|c|}
 \hline
  $n$ & 1 & 2 & 3 & 4 & 5& 6& 7& 8& 9 \\
 \hline
  \hline
 $\gamma$ & 1 & 2 & 2 & 4& 7& 12 & 16& 32 & 62  \\
 \hline
 $\gamma_t$ & 2 & 2 & 4 & 4& 8 & 14 & 24 & 32 & 64\\
 \hline
 $\rho_2$ & 1 & 1 & 2 & 2& 4 & 8 & 16 & 17-30 & ?\\
 \hline
 $\roo$ & 2 & 2 & 2 & 4 & 4 & 8 & 16 & 32 & 34-60 \\
 \hline
\end{tabular}
\end{center}
\caption{Packing and domination invariants in hypercubes $Q_n$, where $n<10$.}
\end{table}

In addition, consider the value $\gamma_t(Q_{2^k}) = 2^{2^k-k}$, which follows from Theorem~\ref{thm:infinite-families} combined with the formula $\gamma_t(G\cp K_2)=2\gamma(G)$ from~\cite{azarija-2017}. Now, compare this with the bound $\roo(Q_{2^k}) \ge 2^{2^k-k}$, which follows from Theorem~\ref{thm:lower-bound}(ii) when plugging $n=2^k$. Since $\gamma_t(G)\ge \roo(G)$ for every graph $G$ with no isolated vertices, we infer that
\begin{equation}
\label{eq:cubes}
\gamma_t(Q_{2^k})=2^{2^k-k}=\roo(Q_{2^k}), \textrm{ for all } k\in \mathbb{N}.
\end{equation}

Recall the result from~\cite{Rall} stating that $\gamma_t(G\times H) = \gamma_t(G)\gamma_t(H)$ whenever $G$ is a graph with $\roo(G)=\gamma_t(G)$ and graphs $G$ and $H$ have no isolated vertices. Therefore, from the discussion above we get that 
$$\gamma_t(Q_{2^k}\times H) = 2^{2^k-k}\gamma_t(H)\,,$$ where $k\in \mathbb{N}$ and $H$ is an arbitrary graph with no isolated vertices. An additional family of graphs with this property (that $\gamma_t=\roo$) can be found in~\cite{mohammadi-2019}. It would be interesting to establish if there are any hypercubes $Q_n$ of other dimensions than those in~\eqref{eq:cubes} that satisfy the equality $\gamma_t(Q_n)=\roo(Q_n)$.

%%%%%%%%%%%%%%%%%%%%%%%%%%%%%
\section*{Acknowledgments}
%%%%%%%%%%%%%%%%%%%%%%%%%%%%%

This work was performed within the bilateral grant `Domination in graphs, digraphs and their products" (BI-US/22-24-038). B.B.\ and S.K.\ were supported by the Slovenian Research Agency (ARRS) under the grants P1-0297, J1-2452, N1-0285, and J1-3002.

\end{document}